\documentclass[12pt]{amsart}
\usepackage{newlfont,amsmath,enumerate,amssymb,bbm}
\usepackage{color}

\newcommand{\bbZ}{{\mathbb{Z}}}

\newcommand{\Hom}{{\mathrm{Hom}}}
\newcommand{\Ind}{{\mathrm{Ind}}}
\newcommand{\ind}{{\mathrm{ind}}}

\newcommand{\SL}{{\mathrm{SL}}}

\newtheorem{thm}{Theorem}[section]

\newtheorem{lemma}[thm]{Lemma}

\newtheorem{prop}[thm]{Proposition}

\newtheorem{cor}[thm]{Corollary}

\def\Z{{\mathbb Z}}
\def\R{{\mathbb R}}

\def\presuper#1#2%
  {\mathop{}%
   \mathopen{\vphantom{#2}}^{#1}%
   \kern-\scriptspace%
   #2}

\begin{document}


\title[]{Bounded contractions for affine buildings}

\author{Mladen Bestvina and Gordan Savin}

\address{Department of Mathematics, University of Utah, Salt Lake City, UT 84112}
\thanks{M.B. partially supported by NSF grant DMS-1607236.} 
\thanks{G.S. partially supported by NSF grant DMS-1359774.} 

\email{bestvina@math.utah.edu}

\email{savin@math.utah.edu}

\begin{abstract} 
We consider affine buildings with refined chamber structure. For each vertex $x$  we construct a contraction, based at $x$, 
 that is used to prove exactness of Schneider-Stuhler resolutions of arbitrary depth. 
\end{abstract}

\maketitle

\section{Introduction} 

Let $k$ be a local non-archimedean field and $G$ the group of
$k$-points of an algebraic reductive group defined over $k$.  Let $X$
be the Bruhat-Tits building attached to $G$, which we view as a
$CAT(0)$ metric
space. In \cite{MP2}, to every $x\in X$ and non-negative real number
$r$, Moy and Prasad attach a subgroup $G_{x,r}$ such that
$G_{x}:=G_{x,0}$ is the parahoric subgroup of $G$ attached to $x$,
$G_{x,s} \subseteq G_{x,r}$ whenever $r\leq s$, and $G_{x,r}$ are
normal subgroups of $G_x$.  For every $r\geq 0$ we let
$G_{x,r,+}=\cup_{s>r} G_{x,s}$. The building $X$ has a (standard)
structure of a chamber complex such that, for every facet $\sigma$ and
$r$ an integer, the function $x\mapsto G_{x,r,+}$ is constant for $x$
in the interior of $\sigma$. Let $r$ be a rational number.  Then the
chamber structure can be refined so that $x\mapsto G_{x,r,+}$ is
constant on the interior points of each facet \cite{BKW}. For example,
if $G=\SL_2(k)$, the building $X$ is a tree and if we divide each edge
into two, of equal lengths, then $x\mapsto G_{x,r,+}$ is constant on
the interior points of each edge if $r$ is half-integral.  Fix $r$ and
a refined chamber structure on $X$.  Let $C_i(X)$ be the free abelian
group with the basis consisting of the $i$-dimensional facets of
$X$. Let $C_{-1}(X)=\mathbb Z$.  Let $x\in X$ be a vertex.  Recall
that a {\it contraction} $c$ of $C_*(X)$ based at $x$ is a sequence of
homomorphisms $c_i:C_i(X)\to C_{i+1}(X)$, $i=-1,0,1,\cdots$, such that
$c_{-1}(1)=x$ and $c_{i-1}\partial+\partial c_i=1$. Our main technical
result is the construction of a contraction of $\sigma\to
c(\sigma)=\sum c(\sigma,\tau)\tau$ with the following properties:
\begin{enumerate} 
\item the contraction $c$ is $G_x$-equivariant. 
\item if $c(\sigma,\tau)\neq 0$ then $\tau$ is contained in the smallest subcomplex of $X$ containing the cone with the vertex $x$ and the base $\sigma$. 
\item the coefficients $c(\sigma,\tau)$ are uniformly bounded. 
\end{enumerate} 
A more detailed meaning of (2) is the following. The cone is the union of all geodesic segments connecting $x$ to a point in $\sigma$. If $c(\sigma,\tau)\neq 0$ then there exists a point 
$y\in \sigma$ and a point $z$ on the geodesic $[x,y]$ such that $z$ is either an interior point of $\tau$ or $\tau$ is in the boundary of a facet containing $z$ as an interior point.

Let $V$ be a smooth representation of $G$ of depth $r$. Following Schneider-Stuhler  \cite{SS} 
we define a projective resolution of $V$ using the  chain complex
$C_*(X)$.   Since we use a refined chamber structure on $X$ the
projective modules in the resolution have the same depth $r$. The
exactness of resolution is a simple consequence of the existence of the contraction $c$ and niceness of Moy-Prasad groups $G_{x,r}$ if $r>0$. 
 An open compact subgroup $K$ of $G$ is nice if, for any smooth 
representation $V$ of $G$ generated by its $K$-fixed vectors, any subquotient of $V$ is generated by its $K$-fixed vectors. It is recorded and used in the literature that $G_{x,r}$ is nice if $x$ is a special point in the building. Niceness of all positive depth Moy-Prasad groups is possibly known to experts, however, we have not found a precise reference in the literature and hence have included a proof. We thank C. Bushnell and G. Henniart for a discussion on this matter. 
The property (3) is not used in this paper, however, it is critical to prove that the resolution stays exact after passing to a Schwartz completion, see \cite{OS} for details.

\section{Lipschitz simplicial approximation}

We fix a cell structure $A$ on $\R^n$ so that each (closed) cell
$\sigma$ is a convex polyhedron, and when $\tau\subset \sigma$ are
cells in $A$ then
$\tau$ is a polyhedral face of $\sigma$. For simplicity, we will
assume that $A$ is invariant under a discrete group of translations of
rank $n$, although the arguments below apply when $A$ has bounded
combinatorics and contains only finitely many isometry types of
cells. We also fix an orientation of each cell.  By $A^i$ we denote
the $i$-skeleton of $A$. Consider the augmented cellular chain complex
$C_*(A)$ of $A$ where $C_i(A)$ is the free abelian group with the
basis consisting of the $i$-cells of $A$, and $C_{-1}(A)=\bbZ$. The
boundary morphisms are defined by
$$\partial \sigma=\sum_\tau \epsilon(\sigma,\tau)\tau$$
where the sum runs over codimension 1 faces $\tau$ of $\sigma$ and
$\epsilon(\sigma,\tau)=1$ if the orientation of $\sigma$ is obtained
from the orientation of $\tau$ by appending an outward normal to
$\tau$, and otherwise $\epsilon(\sigma,\tau)=-1$.

We
fix a vertex $v_0$ of $A^0$. Without loss of generality we can assume
that $v_0=0$.  Recall that a {\it contraction} $c$ of $C_*(A)$ based
at $v_0$ is a sequence of homomorphisms $c_i:C_i(A)\to C_{i+1}(A)$,
$i=-1,0,1,\cdots$, where $c_{-1}(1)=v_0$, such that
$c_{i-1}\partial+\partial c_i=1$. The goal of this section is to
construct $\sigma\to c(\sigma)=\sum c(\sigma,\tau)\tau$ with the
coefficients $c(\sigma,\tau)$ uniformly bounded and
$c(\sigma,\tau)\neq 0$ only if $\tau$ is contained in every convex
subcomplex of $A$ containing $v_0$ and $\sigma$.
  
The standard method of deforming a map $f:\R^k\to \R^n$ so it misses a
given cell $\sigma$ of dimension $>k$ is to radially project to
$\partial\sigma$ from a
point $x_\sigma$ in the interior of $\sigma$ and in the complement of the image
of $f$. For the purposes of controlling the coefficients
$c(\sigma,\tau)$ in the contraction we need to control the Lipschitz
constant, which will dramatically increase if $x_\sigma$ is close to
the image of $f$.

\begin{prop}\label{1}
For every $N,L>0$ there is $L'>0$ with the following properties. Let
$D\subset \R^k$ be a polyhedral cell and $f:D\to \R^n$ a
map such that 
\begin{itemize}
\item $f$ is $L$-Lipschitz,
\item $f(\partial D)\subset A^{k-1}$,
\item $f(D)\cap\sigma$ is contained in the union of $N$ $k$-planes for
  every cell $\sigma$ of $A$.
\end{itemize}
Then there is a map $g:D\to \R^n$ such that
\begin{enumerate}[(i)]
\item $g=f$ on $\partial D$, and on $f^{-1}(A^{k-1})$,
\item $g$ is $L'$-Lipschitz,
\item for every $x\in D$, $g(x)$ belongs to the closure of the open
  cell of $A$ that contains $f(x)$,
\item $g(D) \subset A^k$. 
\end{enumerate}
\end{prop}

\begin{proof}
We inductively construct maps $f=g_n,g_{n-1},\cdots,g_k=g:D\to\R^n$
satisfying (i), (iv), (modified) (ii)-(iii): $g_i(D)\subset A^i$,
$g_i$ is $L_i$-Lipschitz, and in addition $g_i(D)\cap\sigma$ is
contained in the union of $N_i$ $i$-planes. The construction from
$g_i$ to $g_{i-1}$ is as follows.

By volume considerations, there is $\epsilon>0$ that depends only on
$N_i$ and $A$, so that for every $i$-cell $\sigma$ there is a point
$x_\sigma\in \overset{\circ}\sigma$ at distance $>\epsilon$ from
$g_i(D)\cap\sigma$ and also distance $>\epsilon$ from
$\partial\sigma$. Compose $g_i$ with the map which is the radial
projection from $x_\sigma$ to $\partial\sigma$ on $\sigma$ for each
$i$-simplex $\sigma$. This composition is $g_{i-1}$. The Lipschitz
constant $L_{i-1}$ of $g_{i-1}$ is uniformly bounded, more precisely, 
\[ 
L_{i-1} \leq \left(\frac{d}{\epsilon}\right)^2 L_{i}
\] 
where $d$ is a bound on the diameter of all simplices, and
$N_{i-1}\leq N_iM_iP_i$, where $M_i$ is the maximal number of
codimension 1 faces of a cell in $A$ and $P_i$ is the maximal number
of cells intersecting a given cell.
\end{proof}

\begin{prop} \label{P:lip} 
Let $v_0=0$ be a vertex of $A$. For every cell $\sigma$ of $A$
there is a homotopy $H_\sigma:\sigma\times [0,D_\sigma]\to\R^n$, with
$D_\sigma=\mathrm{diam} (\sigma\cup\{v_0\})$, so that the following holds.
\begin{enumerate}[(1)]
\item $H_\sigma(x,0)=v_0$ for every $x\in\sigma$,
\item $H_\sigma(\cdot,1)$ is the inclusion,
\item the image of $H_\sigma$ is contained in $A^{\dim(\sigma)+1}$,
  and it is also contained in every convex subcomplex of $A$ that
  contains $v_0$ and $\sigma$,
\item the restriction of $H_\sigma$ to $\tau\times [0,D_\sigma]$ for a
  face $\tau<\sigma$ is the linear reparametrization of $H_\tau$,
\item $H_\sigma$ is uniformly Lipschitz,
\item $H_\sigma$ is uniformly bounded distance away from the straight
  line homotopy $(x,t)\mapsto tx/D_\sigma$.
\end{enumerate}
\end{prop}

In the proof we will need the following lemma.

\begin{lemma}\label{cone}
For every $L_0,\epsilon,D,\alpha_0>0$ there exists $L'>0$ such that
the following holds. Let $\sigma\subset \R^k$ be a convex polyhedral
cell, $v\in \R^k$, $f:\sigma\to\R^n$, $w\in\R^n$ such that:
\begin{enumerate}[(a)]
\item $f$ is $L$-Lipschitz,
\item $|v-x|>\epsilon$ for every $x\in\sigma$,
\item $|w-y|\leq D$ for every $y\in f(\sigma)$,
\item $\angle_y(x,v)>\alpha_0$ for every $x,y\in\sigma$, $x\neq y$.
\end{enumerate}
Then the map $F:v*\sigma\to\R^n$ defined on the cone
by $$F((1-t)v+tx)=(1-t)w+tf(x)$$ is $L'$-Lipschitz.
\end{lemma}

\begin{proof}
We may assume $v=w=0$. The extension $F$ is clearly $L$-Lipschitz when
restricted to each slice $t\sigma$ by (a), and by (b) and (c) $F$ is also
uniformly Lipschitz when restricted to each radial line $\{tx\mid t\in
[0,1]\}$ for any $x\in\sigma$. Any two points in the cone $v*\sigma$
can be connected by a segment in a slice followed by a segment in a
radial line. By (d) the angle between these segments is bounded below,
so the total length is bounded by a fixed multiple of the distance
between the two points, implying the result.
\end{proof}

\begin{proof}[Proof of Proposition \ref{P:lip}]
The construction is by induction on $\dim(\sigma)$. When $\sigma$ is a
vertex apply Proposition \ref{1} to the geodesic $f$ joining $v_0$ to
$\sigma$. This gives $g=H_\sigma:[0,D_\sigma]\to \R^n$.

Assume now that $\dim(\sigma)=k$ and $H_\tau$ has been constructed for
all faces $\tau<\sigma$.  Define $f:\partial(\sigma\times
[0,D_\sigma])\to \R^n$ to be constant $v_0=0$ on $\sigma\times \{0\}$,
inclusion on $\sigma\times\{D_\sigma\}$, and linearly reparametrized
$H_\tau$ on $\tau\times [0,D_\sigma]$, for every face
$\tau<\sigma$. We now wish to extend $f$ to all of $\sigma\times
[0,D_\sigma]$ and apply Proposition \ref{1} to this extension. If we
extend by coning off from a point in the interior of $\sigma\times
[0,D_\sigma]$, the Lipschitz constant might blow up since the angles
as in Lemma \ref{cone}(d) will be small when $D_\sigma$ is
large. Instead, we first subdivide $[0,D_\sigma]$ into
$0=t_0<t_1<\cdots<t_\sigma=D_\sigma$ so that the length $t_i-t_{i-1}$
of each segment belongs to a fixed interval $[B_1,B_2]$ with
$B_1>0$. Then we use Lemma \ref{cone} to extend $f$ to $\sigma\times
\{t_i\}$ for each $i$. For this we use $v_i=(x_\sigma,t_i)$ for a
point $x_\sigma$ in the interior of $\sigma$, and we use the same
point for this isometry type. The image $f(\partial
\sigma)\times\{t_i\}$ has uniformly bounded diameter, and we can set
$w_i\in\R^n$ to be any point in this image. Because of (5) and (6),
the assumptions of Lemma \ref{cone} are satisfied and we get uniformly
Lipschitz extensions of $f$. It remains to extend $f$ to each
$\sigma\times [t_{i-1},t_i]$ this is done in exactly the same way, by
coning from $(x_\sigma,\frac{t_{i-1},t_i}2)$ with respect to a point
in the image of $\partial(\sigma\times [t_{i-1},t_i])$. Thus the
extension $\sigma\times [0,D_\sigma]\to \R^n$ is uniformly
Lipschitz. The third bullet in Proposition \ref{1} holds since only a
bounded number of images of $\sigma\times [t_{i-1},t_i]$ intersect a
given cell. Thus we can apply Proposition \ref{1} to this extension to
get the desired map $H_\sigma$.
\end{proof}

When $\sigma$ is a $k$-cell and $\tau$ a $(k+1)$-cell (both oriented)
we denote by $c(\sigma,\tau)$ the degree of the map
$$(\sigma\times [0,D_\sigma],\partial (\sigma\times
[0,D_\sigma]))\overset {H_\sigma} \to (A^{k+1},A^{k+1}\smallsetminus
\overset\circ \tau)$$
(the latter set is by excision equivalent to $(\tau,\partial\tau)$). 

\begin{prop}
The numbers $c(\sigma,\tau)\in\Z$ are uniformly bounded.
\end{prop}

\begin{proof}
A warmup is the fact that any map $S^n\to S^n$ of large degree must
have a large Lipschitz constant. This follows from Arzela-Ascoli. The
set of maps $S^n\to S^n$ with Lipschitz constants bounded above is
compact and hence represents finitely many homotopy classes.

The proof in our relative case is similar. Suppose there is a sequence
of pairs $(\sigma_i,\tau_i)$ with $|c(\sigma_i,\tau_i)|\to\infty$. We
may assume all $\sigma_i$ are isometric and have dimension $k$, and
all $\tau_i$ are isometric. By (5) the preimage of $\tau_i$ under
$H_{\sigma_i}$ is contained in $\sigma_i\times [u_i,u_i+C]$ for some
fixed $C$ independent of $i$. The image of $\sigma_i\times
[u_i,u_i+C]$ is contained in a subcomplex $Y_i\subset A^{k+1}$ of
$\tau_i$ and up to isomorphism there are only finitely many
possibilities.

After translating the interval and identifying all $\sigma_i$ and all
$\tau_i$ and the $Y_i$ as above we have a sequence of maps
$\sigma\times [0,C]\to Y$, and these all have uniformly bounded
Lipschitz constants. By Arzela-Ascoli after a subsequence they will be
close to each other and will determine the same degree. Contradiction.
\end{proof}

Now $\sigma\to c(\sigma)=-\sum c(\sigma,\tau)\tau$ defines a
contraction.

Assume that $X$ is a building corresponding to a reductive group $G$ over a $p$-adic field $k$, with a refined chamber structure i.e. we have divided chambers into smaller ones in 
a $G$-equivariant fashion. Let $x$ be a vertex in $G$ and $G_x$ the (largest) parahoric group fixing $x$.  
 For every facet $\sigma \subset X$  the cone with the vertex $x$ and the base $\sigma$ is  contained in any apartment $A$ containing $x$ and $\sigma$. 
 Thus the contraction $c(\sigma)$ can be defined working in any such apartment. If $g\in G_x$ stabilizes $\sigma$ then $g$ fixes $\sigma$  point-wise, hence also the cone and $c(\sigma)$. Hence 
 $c$ can be extended to whole $G_x$-orbit of $\sigma$ in an $G_x$-equivariant fashion, and we have proved: 
 
 \begin{prop} For every vertex $x$ in $X$ there exists a $G_x$-invariant contraction $c$ satisfying the conditions (1)-(3) in Introduction.  
 \end{prop}

\section{Moy-Prasad groups} 
Let $k$ be a non-archimedean local field and $G$ the group of $k$ points of a reductive algebraic group defined over $k$. Let $S$ be a maximal $k$-split torus in $G$ and 
$\Phi=\Phi(G,S)$ the corresponding restricted root system. Let $T$ be the centralizer of $S$ in $G$. 
A decomposition $\Phi=\Phi^+ \cup \Phi^-$ of the root system into positive and negative roots defines a pair of maximal unipotent subgroups $U$ and $\bar U$ of $G$.  
Let $A(S)$ be the apartment in the building $X$ of $G$ stabilized by $S$. 
Let $x\in A(S)$ and $r>0$. The Moy-Prasad group $G_{x,r}$ has an Iwahori decomposition (\cite{MP2} Theorem 4.2) 
\[ 
G_{x,r}= \bar U_{x,r} T_r U_{x,r} 
\] 
where the three factors are the intersection of $G_{x,r}$ with $\bar U$, $T$ and $U$, respectively. The factor $T_r$ is independent of $x$ as indicated.  

 \begin{lemma} \label{L:MP1} Let $x, y\in X$ and $z$ be a point on the geodesic connecting $x$ and $y$. If $r>0$ then 
 \[ 
  G_{z,r}\subseteq  G_{x,r} G_{y,r}. 
 \]   
 \end{lemma}
 \begin{proof} Without loss of generality we can assume that $x,y\in A(S)$.  The apartment $A(S)$ is an affine space so $v=x-y$ is a well defined vector in the space of translations 
 of $A(S)$. Elements of $\Phi$ (roots) are functionals on the space of translations of $A(S)$. Thus $\alpha(v)$ is a well defined real number for every $\alpha\in \Phi$. Pick a 
 decomposition $\Phi=\Phi^+ \cup \Phi^-$  such that $\alpha(v) \geq 0$ for all $\alpha\in \Phi^+$. Then 
 \[ 
\bar U_{x,r} \supseteq \bar U_{z,r} \text{ and } U_{y,r} \supseteq U_{z,r}. 
\] 
This inclusions are a direct consequence of the definition of $G_{x,r}$ \cite{MP1} if $G$ is quasi-split. The general case is then deduced by checking it over a finite unramified extension of $k$, over which $G$ is quasi-split, and then taking fixed points for the Galois action. Lemma follows from the Iwahori decomposition of $G_{z,r}$. 
 \end{proof} 
 
The Iwahori decomposition and Lemma  \ref{L:MP1} hold for the Groups $G_{x,r,+}$ (for $r\geq 0$) since $G_{x,r,+}=G_{x,s}$ for all $s>r$ and sufficiently close to $r$. 
Now fix a non-negative rational number $r$. Refine the chamber decomposition of $X$ so that the function $x\mapsto G_{x,r,+}$ is constant on interiors of facets. Thus, for any 
facet $\sigma$ we define $K_{\sigma}= G_{x,r,+}$ where $x$ is any interior point of $\sigma$. If $\tau$ is a facet in the boundary of $\sigma$ then $K_{\tau} \subseteq K_{\sigma}$. 
The refinement of $X$ and a check of these inclusions is done explicitly for quasi-split groups. We shall give details in the split case below. 
The general case follows by unramified Galois descent. 

 \begin{lemma} \label{L:MP2}  Let $x$ be a vertex of $X$ and $\sigma$ a facet for a refined decomposition of $X$. Assume that a facet 
 $\tau$ is contained in the smallest subcomplex of $X$ containing the cone with the vertex $x$ and the base $\sigma$. Then 
 \[ 
 K_{\tau} \subseteq K_x  K_{\sigma} 
 \] 
  \end{lemma}  
 \begin{proof} By the assumption, there exists a point $y\in \sigma$ and a point $z$ on the geodesic $[x,y]$ such that the facet containing $z$,  as an interior point, contains $\tau$. 
 Then $K_{\tau} \subseteq G_{z,r,+}$ and $G_{y,r,+}\subseteq K_{\sigma}$.  Lemma follows from Lemma  \ref{L:MP1}. 
  \end{proof}

\vskip 10pt 
We shall work out the details in the case $G$ is simple and split. (A similar treatment for quasi split groups can be easily given based on computations in \cite{PR}.) 
Since $G$ is split, $S$ is a maximal torus, hence $T=S$. 
Let  $O$ be the ring of integers in $k$ and $\pi$ and a uniformizing element. The group $k^{\times}$ has a natural filtration 
$ k^{\times} \supset O^{\times} \supset 1+ \pi O \supset 1+\pi^2O \supset \ldots $ which gives rise to a filtration 
  \[ 
  T\supset T_0 \supset T_1 \supset \ldots 
  \] 
   where $T_0$ is the maximal compact subgroup, which is the same as the set of all $t\in T$ such that $\chi(t)\in O^{\times}$ for all algebraic characters 
 $\chi$ of $T$,  and $T_r$ is the set of all $t\in T$ such that $\chi(t)\in 1+\pi^rO$ for all algebraic characters $\chi$. 
 The apartment $A=A(S)$ can be identified with $\Hom(\mathbb G_m, S)\otimes \mathbb R$, so $0$ is a special vertex of $A$. The apartment is a Coxeter 
 complex defined by an affine root system $\Psi= \cup_{n\in \mathbb Z} \Phi + n $. 

For every root $\alpha\in \Phi$ we have a subgroup $U_{\alpha}$ of $G$ isomorphic to $k$, and 
 for every affine root $a$ we have a subgroup isomorphic to $O$. More precisely, 
 If $a$ is an affine root whose gradient is $\alpha$, then $U_a$ is defined as the subgroup of $U_{\alpha}$ consisting of elements fixing the half-plane $a\geq 0$.
 The filtration 
\[ 
\ldots \supset U_{a-1}\supset U_a \supset U_{a+1} \supset \ldots 
\] 
of $U_{\alpha}$ corresponds to the filtration $\ldots \supset \pi^{-1} O \supset O \supset \pi O \supset \ldots $ of $k$. 
For example, if $G=\SL_2(k)$ and $T$ the torus of diagonal matrices, then one can identify $A$ with $\mathbb R$ and the set of affine 
roots with the set of affine functions $x\mapsto \pm x+n$, where $n\in\mathbb Z$, so that for $a(x)=x+n$ 
\[ 
U_a=\left\{\left[\begin{array}{cc} 
1 & x \\
0 & 1 \\
\end{array} \right] | ~ x\in \pi^n O\right\} \text{ and } 
U_{-a}=\left\{\left[\begin{array}{cc} 
1 & 0 \\
x & 1 \\
\end{array} \right] | ~ x\in \pi^{-n} O\right\}
\] 

It is convenient to introduce a notion of imaginary roots. The imaginary roots are constant, integer-valued functions on $A$. Let 
\[ 
\tilde\Psi=\Psi \cup \mathbb Z. 
\] 
If $a\in \tilde\Psi$ is an imaginary non-negative root, we define $U_a=T_a\subseteq T$. Let $x\in A$ and $r\geq 0$, a real number. 
Then $G_{x,r}$ and $G_{x,r,+}$ are generated by $U_a$ for all $a\in \tilde \Psi$  such that $a(x) \geq r$ and $a(x) > r$, respectively.   If we 
fix a decomposition $\Phi=\Phi^+ \cup \Phi^-$ then the groups $U_{x,r}$ and $\bar U_{x,r}$, appearing in the Iwahori decomposition, are generated by $U_a$ such that $a(x)\geq r$ and the gradient of 
$a$ is in $\Phi^+$ and $\Phi^-$, respectively. Now the inclusions in the proof of Lemma \ref{L:MP1} are clear.  

If $r$ is an integer and $a(x)>r$ for one interior point of a facet $\sigma$ then the same is true for any point in the interior of $\sigma$. Thus the 
groups $G_{x,r,+}$ are constant on interiors of facets. Moreover, if $\tau$ is contained in the boundary of $\sigma$, and $a> r$ on the interior of 
$\tau$ then it is so on the interior of $\sigma$. Hence $K_{\tau} \subseteq K_{\sigma}$. Finally, assume $r=\frac{n}{m}$ is a non-negative rational number. We 
can refine the Coxeter complex by replacing $\Psi$ with $\Psi_m= \cup_{n\in \mathbb Z} \Phi + \frac{n}{m}$. Now the groups $G_{x,r,+}$ are constant on interiors of facets.

 \section{Schneider-Stuhler complex}

 We fix  a non-negative rational number $r$ throughout this section. We refine the chamber decomposition of $X$ so that the groups $G_{x,r,+}$ are constant on the interior of 
 facetes and we set $K_{\sigma}:=G_{x,r,+}$ where $x$ is any interior point of a facet $\sigma$. 
 Let $V$ be a smooth representation of $G$, this means that any vector
 $v\in V$ is fixed by an open compact subgroup of $G$, depending on
 $v$. For every facet $\sigma$  let $V_{\sigma}$ be the subspace of all vectors in $V$  fixed by $K_{\sigma}$.  
The complex $C_*(X) \otimes_{\mathbb Z} V$ admits a natural representation of $G$ defined by $g(\sigma\otimes v)= g(\sigma)\otimes g(v)$, for all $g\in G$. 
  Let $C_*(X,V)$ be the subcomplex  spanned by  $\tau \otimes v$ where $v\in V_{\tau}$. 
 The boundary $\partial$ preserves $C_*(X,V)$ because $V_{\sigma}
 \subseteq V_{\rho}$ any time $\rho$ is in the boundary of $\sigma$. The action of $G$ on $C_*(X)\otimes V$ preserves the sub complex $C_*(X, V)$.  
 
 \begin{thm}  Let $x$ be a vertex in $X$ and $c$ an $x$-based contraction of $C_*(X)$ satisfying the conditions (1) and (2) in the Introduction. 
 Then the complex $C_*(X,V)^{K_{x}}$ is exact. 
 \end{thm} 
 \begin{proof} 
 The contraction $c$ defines a contraction on $C_*(X)\otimes V$ by $c(\sigma\otimes v)=c(\sigma)\otimes v$. To prove the theorem, it 
  suffices to show that the contraction preserves the subcomplex $C_*(X,V)^{K_{x}}$. 
  Let  $e_{x}: V \rightarrow V_{x}$ be the projection 
 given by averaging the action of $K_{x}$ on $V$ with respect to a Haar measure on $K_x$ of volume one. 
   The subcomplex  $C_*(X,V)^{K_{x}}$ is spanned by elements $e_x(\sigma\otimes v)$ where $v\in V_{\sigma}$. Since $c$ is $K_{x}$-invariant, 
 \[ 
 c(e_{x}(\sigma\otimes v )) = e_{x} (c(\sigma)\otimes v) =\sum_{\tau} c(\sigma,\tau) e_x(\tau\otimes v). 
 \] 
 By the property (2) and Lemma \ref{L:MP2},  for every $\tau$ appearing in this sum, there exist $g_1, \ldots g_n\in K_{x}\cap K_{\tau}$ such that 
 \[ 
 K_{\tau}= \cup_{i=1}^n g_i(K_{\sigma}\cap K_{\tau}) 
 \] 
 (a disjoint sum). Since $g_i\in K_{x}$, it follows that $e_x \cdot g_i =e_x$, as operators on any representation of $G$.  Hence 
  \[ 
 e_x(\tau\otimes v)= \frac{1}{n}\sum_{i=1}^n e_x g_i ( \tau \otimes  v) . 
 \]  
Since $g_i\in K_{\tau}$, these elements fix $\tau$.  Hence 
 \[ 
 e_x(\tau\otimes v)= e_x( \tau \otimes  \frac{1}{n}\sum_{i=1}^ng_i v). 
 \]  
 Since $v$ is fixed by $K_{\sigma}\cap K_{\tau}$ and $g_i$ are representatives of all $K_{\sigma}\cap K_{\tau}$-cosets in $K_{\tau}$, it follows that 
 \[ 
 \frac{1}{n}\sum_{i=1}^ng_i v=e_{\tau} (v)\in V_{\tau} 
 \] 
 where  $e_{\tau}: V \rightarrow V_{\tau}$ is the projection 
 given by averaging the action of $K_{\tau}$ on $V$ with respect to a Haar measure on $K_{\tau}$ of volume one. Hence $e_x(\tau\otimes v)\in C_*(X,V)^{K_{x}}$ as desired. 
\end{proof} 

\section{Nice open compact subgroups}

 Let $S$ be a maximal split torus in $G$. Let $P$ be a parabolic subgroup of $G$. Without loss of generality we shall assume that $P$ contais $S$.
  In particular, we have a ``standard" choice of the Levi $L$ and the opposite $\bar U$ of the radical $U$ of $P$, both
normalized by $S$.  Let $K$ be a an open compact subgroup of $G$ and $K^G$ the set of 
all $G$-conjugates of $K$. The parabolic group $P$ acts on $K^G$ with finitely many orbits. We say that  $K$ is nice with respect to $P$ if in any $P$-conjugacy class in 
$K^G$ there is $K'$ such that the Iwahori decomposition holds: 
\[ 
K'=(K'\cap U)(K'\cap L)(K'\cap \bar U). 
\] 
If that is the case, then $K'_L=K'\cap L$ is isomorphic to $K'\cap P/K'\cap U$. This observation implies the first of the following properties, for the second see 
 Proposition 3.5.2. in \cite{De}.

\begin{enumerate} 
\item If $W$ is an $L$-module and $(\Ind_P^G W)^K\neq 0$ then $W^{K'_L}\neq 0$ for some $K'$ in $K^G$ with the Iwahori decomposition.  
\item For any $G$-module $V$, and $K'$ in $K^G$ with the Iwahori decomposition, the map $V^{K'} \rightarrow V_{U}^{K'_L}$ is surjective. 
\end{enumerate} 
We say that $K$ is nice if it is nice with respect to any parabolic $P$ containing $S$. 

\begin{prop} Assume that $K$ is nice and $V$ is a $G$-module generated by $V^K$. If $V'$ is a non-trivial subquotient of $V$ then $(V')^K\neq 0$. 
\end{prop} 
\begin{proof} 
It suffices to prove the statement for irreducible subquotients. Without loss of generality we can assume that $V$ is contained in a single Bernstein component 
corresponding to a pair $(L, W)$ where $L$ is a Levi group containing $S$ and $W$ a quasi-cuspidal representation of $L$. Let $V'$ be an irreducible subquotient of $V$. Then there 
exists an unramified twist $W'$ of $W$ such that $V'$ is a submodule of $\Ind_P^G W'$ where $P$ is a parabolic subgroup containing $L$. 
Since $V^K\neq 0$ there exists at least one irreducible subquotient $V_0$ such that $(V_0)^K\neq 0$.   Let $W_0$ be an unramified twist of $W$ such that 
$V_0$ is contained in $\Ind_P^G W_0$. It follows that $(\Ind_P^G W_0)^K\neq 0$. 
By (1) above, there exists $K'\in K^G$ 
with the Iwahori decomposition such that $(W_0)^{K'_L}\neq 0$. Hence $(W')^{K'_L}\neq 0$ for any unramified twist $W'$ of $W$. Now let $V'$ be any irreducible quotient 
of $V$. From the Frobenius reciprocity 
\[ 
\Hom_G(V', \Ind_P^G (W'))=  \Hom_L(V'_U, W') 
\] 
it follows that $W'$ is a quotient of $V'_U$. Hence 
$(V'_U)^{K'_L}\neq 0$ and $(V')^{K'}\neq 0$ by (2) above. Hence $(V')^K\neq 0$ since $K'$ is conjugate to $K$. 
\end{proof}

\begin{prop}  For every $x\in X$ and $r>0$ the Moy-Prasad group $G_{x,r}$ is nice. 
\end{prop} 
\begin{proof} Without loss of generality we can assume that $x\in A(S)$.  Let $P$ be a parabolic subgroup containing $S$. Let $\sigma$ be a chamber in $A(S)$  containing $x$. Then the point-wise stabilizer $G_{\sigma}$ of $\sigma$ is an Iwahori subgroup of $G$.  Let $N$  be the normalizer of $S$ in $G$. 
It acts naturally on $A(S)$.  We have an Iwasawa decomposition (\cite{BT1}, Proposition 7.3.1)
\[ 
 G=P N G_{\sigma}. 
\] 
Since $G_{\sigma}\subset G_{x}$, the Iwahori group $G_{\sigma}$ normalizes $G_{x,r}$, hence representatives of $P$-orbits in the $G$-conjugacy class of $K=G_{x,r}$ can be taken to 
be $K'=G_{x',r}$ where $x'=n(x)\in A(S)$ for some $n\in N$. And these groups have the Iwahori decomposition with respect to $P$, since $P$ contains $S$ by Theorem 4.2 in \cite{MP2}. 
\end{proof} 

\begin{cor} Assume that, for every vertex $x\in X$ there exists a contraction of $C_*(X)$  satisfying the properties (1) and (2). Then the Schneider-Stuhler complex is 
exact for every smooth representation $V$. 
\end{cor} 
\begin{proof} It suffices to prove that the complex is exact in every Bernstein component. 
The complex $C_*(X,V)$ is a direct sum of $G$-modules isomorphic to $\ind_{S_{\tau}}^G  V_{\tau}$ where $\tau$  is a facet in the refined chamber complex and $S_{\tau}$ is the 
stabilizer of $\tau$.  This module is 
generated by $V^{K_{\tau}}=V_{\tau}$.  Let $x$ be a vertex of $\tau$. Since 
$K_x \subseteq K_{\tau}$,  it follows that  $\ind_{G_{\tau}}^G  V_{\tau}$ is generated $K_x$ fixed vectors. 
Thus any Bernstein summand of $C_*(X,V)$ is generated by $K_x$-fixed vectors, for some vertex $x$ and 
exactness can be checked by passing to $K_x$-fixed vectors. 
\end{proof}

\end{document}